\documentclass[reqno]{amsart}%
\usepackage{amsmath}
\usepackage{amsfonts}
\usepackage{CJK}
\usepackage{amssymb}
\usepackage{graphicx}%
\usepackage{hyperref}
\usepackage{amscd}
\usepackage{mathrsfs}
\usepackage{bbm}

\newtheorem{thm}{Theorem}[section]

\newtheorem{lem}[thm]{Lemma}
\newtheorem{prop}[thm]{Proposition}

\theoremstyle{definition}

\theoremstyle{Conjecture}

\theoremstyle{remark}
\newtheorem{rem}{Remark}[section]
\theoremstyle{Example}

\newcommand{\be}{\begin{equation}}
\newcommand{\ee}{\end{equation}}
\newcommand{\bea}{\begin{eqnarray}}
\newcommand{\eea}{\end{eqnarray}}
\newcommand{\ben}{\begin{eqnarray*}}
\newcommand{\een}{\end{eqnarray*}}
\newcommand{\bet}{\begin{equation}
\begin{split}}
\newcommand{\eet}{\end{split}
\end{equation}}

\begin{document}
\title[Characterization of multiplier ideal sheaves]
{Multiplier ideal sheaves with weights of \\ log canonical threshold one}

\author{Qi'an Guan}
\address{Qi'an Guan: School of Mathematical Sciences, and Beijing International Center for Mathematical Research,
Peking University, Beijing, 100871, China.}
\email{guanqian@amss.ac.cn}
\author{Zhenqian Li}
\address{Zhenqian Li: School of Mathematical Sciences, Peking University, Beijing, 100871, China.}
\email{lizhenqian@amss.ac.cn}

\thanks{The first author was partially supported by NSFC-11522101 and NSFC-11431013.}

\date{\today}

\begin{abstract}
In this article, we will characterize the multiplier ideal sheaves with weights of log canonical threshold one by restricting the weights to complex regular surface.
\end{abstract}

\maketitle

\section{Introduction}\label{sec:introduction}

\subsection{Background}
Let $\Omega\subset\mathbb{C}^n$ be a domain and $o\in\Omega$ the origin. Let $u$ be a plurisubhamonic function on $\Omega$. The multiplier ideal sheaf $\mathscr{I}(u)$ is defined to be the sheaf of germs of holomorphic functions $f$ such that $|f|^2e^{-2u}$ is locally integrable (see \cite{De10}). Here, $u$ is regarded as the weight of $\mathscr{I}(u)$. The \emph{Lelong number} of $u$ at $o$ is defined to be
$$\nu(u,o):=\sup\{\gamma\ge0|u(z)\leq\gamma\log|z|+O(1)\ \mbox{near}\ o \}.$$

For Lelong number $\nu(u,o)<1$, Skoda presented $\mathscr{I}(u)_{o}=\mathcal{O}_n$ (see \cite{De10}). When $\nu(u,o)=1$, Favre and Jonsson \cite{F-M} used the valuative tree to characterize the structure of $\mathscr{I}(u)_{o}$ in dimension two. In \cite{G-Z_Lelong1}, Guan and Zhou obtained the result for any dimension $n$ by their solution to Demailly's strong openness conjecture \cite{G-Z_open}.

\begin{thm} \emph{(\cite{G-Z_Lelong1}).} \label{G-Z}
Let $u$ be a plurisubharmonic function on $\Omega\subset\mathbb{C}^{n}$ with $c_{o}(u)=1$. If $\nu(u,o)=1$, then $\mathscr{I}(u)_{o}=\mathscr{I}(\log|h|)_{o}$, where $h$ is the minimal defining function of a germ of regular complex hypersurface through $o$.
\end{thm}

The \emph{log canonical threshold} (or \emph{complex singularity exponent}) of $u$ at $o$ is defined to be $$c_{o}(u):=\sup\{c\ge0|\exp(-2cu)\mbox{ is integrable near}\ o \}.$$
It is convenient to put $c_{o}(-\infty)=0$ (see \cite{D-K}).


Note that $c_o(u)\cdot\nu(u,o)=1$ when $n=1$. In other words, Theorem \ref{G-Z} is equivalent to that if there exists a complex line $L$ through $o$ such that $c_{o'}(u|_L)=1$, then $\mathscr{I}(u)_{o}=(z_1)\cdot\mathcal{O}_n\ \mbox{or}\ \mathcal{O}_n$, in some local coordinates near $o$.

Thus, it's natural to ask\\

\textbf{Question 1.1.} \emph{Let $u$ be a plurisubharmonic function on $\Omega\subset\mathbb{C}^{n}$ and $H$ an $m$-dimensional complex plane in $\mathbb{C}^n$ through $o$. If $c_{o'}(u|_H)=1$, what is the structure of $\mathscr{I}(u)_{o}$?}\\

For the case of $m=1$, Question 1.1 is solved by Theorem \ref{G-Z}. In the following subsection, we will answer the Question for the case of $m=2$.

\subsection{The answer to Question 1.1 for $m=2$ case}
Firstly, we present the following characterization of multiplier ideal sheaves with weights of log canonical threshold one in the two-dimensional case.

\begin{thm} \label{main}
Let $o\in\Omega\subset\mathbb{C}^{2}$ be the origin of $\mathbb{C}^2$ and $u$ a plurisubharmonic function on $\Omega$ with $c_{o}(u)=1$. Then $\mathscr{I}(u)_{o}$ is one of the following cases:

$(1)$ $\mathfrak{m}_{2}$; $(2)$ $(z_1)\cdot\mathcal{O}_{2}$;
$(3)$ $(z_1z_2)\cdot\mathcal{O}_{2}$, up to a change of variables at $o$,
where $\mathfrak{m}_{2}$ is the maximal ideal of $\mathcal{O}_{2}$.
\end{thm}

Dependent on Theorem \ref{main}, we answer Question 1.1 for $m=2$ case.

\begin{thm} \label{main2}
Let $o=(o',o'')\in\Omega\subset\mathbb{C}^{2}\times\mathbb{C}^{n-2}\ (n\ge3)$ be the origin of $\mathbb{C}^n$, $u$ a plurisubharmonic function on domain $\Omega$ and $H{=}\{z_3=\cdots=z_n=0\}\subset\Omega$ a two-dimensional plane through $o$. If $c_{o'}(u|_H)=1$, then there exists a new local coordinates $(w_1,w_2,z_3,...,z_n)$ near $o$ such that $\mathscr{I}(u)_{o}$ is one of the following four cases:

$(1)\ \mathcal{O}_{n}$;\quad\quad $(2)\ (w_1,w_2)\cdot\mathcal{O}_{n}$;\quad\quad
$(3)\ (w_1)\cdot\mathcal{O}_{n}$;

$(4)\ (w_1^2+w_2^2+f(z_3,...,z_n))\cdot\mathcal{O}_{n}$, where $f(z_3,...,z_n)\in\mathfrak{m}_{n}$, the maximal ideal of $\mathcal{O}_{n}$.
\end{thm}

\begin{rem}
$(1)$ All the cases in the above theorem will occur when we take $u=\log(|z_1|^2+|z_2|^2+|z_3|^{\alpha})\ (\alpha>0)$ or $\log(|z_1|^2+|z_2|^2)$ or $\log|z_1|$ or $\log|z_1^2+z_2^2+f(z_3,...,z_n)|$, respectively.

$(2)$ By an algebraic discussion, we obtain that $o$ is a rational singularity of the complex model space $V(w_1^2+w_2^2+f(z_3,...,z_n))$, where $0\not\neq f(z_3,...,z_n)\in\mathfrak{m}_{n}$.
\end{rem}

\section{Proof of main results}

\subsection{Proof of Theorem \ref{main}}
For the proof of main results, we need the following results.

\begin{lem} \emph{(\cite{G-Z_restriction}, see also \cite{De15}).}\label{Guan}
If $c_{o}(u)=1$, then $\mathscr{I}(u)_{o}=(\mathscr{I}_{N(\mathscr{I}(u))})_{o}$, where $\mathscr{I}_{N(\mathscr{I}(u))}$ is the ideal sheaf of the analytic set $N(\mathscr{I}(u))$.
\end{lem}

\begin{thm} \emph{(see \cite{De10}, Theorem 2.18).} \label{Siu_decomposition}
Let $T$ be a closed positive current of bidimension $(p, p)$. Then $T$ can be written as a convergent series of closed positive currents $$T=\sum_{k=1}^{+\infty}\lambda_k[A_k]+R,$$
where $[A_k]$ is a current of integration over an irreducible analytic set of dimension $p$, and $R$ is a residual current with the property that $\dim E_c(R)<p$ for every $c>0$. This decomposition is locally and globally unique: the sets $A_k$ are precisely the $p$-dimensional components occurring in the upperlevel sets $E_c(T)$, and $\lambda_k=\min_{x\in A_k}\nu(T,x)$ is the generic Lelong number of $T$ along $A_k$.
\end{thm}

We are now in a position to prove Theorem \ref{main}.\\
\\
\textbf{\emph{Proof of Theorem} \ref{main}.} By Theorem \ref{G-Z}, it is sufficient to assume $\nu(u,x_0)>1$.

When $\dim_{o}N(\mathscr{I}(u))=0$, by Lemma \ref{Guan}, we have $\mathscr{I}(u)_{o}=\mathfrak{m}_{2}$.

When $\dim_{o}N(\mathscr{I}(u))=1$. As $(N(\mathscr{I}(u)),o)\subset(\{z|\nu(u,z)\ge1\},o)$, we have $\dim_{o}\{z|\nu(u,z)\ge1\}=1$. Let $$(\{z|\nu(u,z)\ge1\},o)=\bigcup_{k=1}^{m}(A_k,o)$$
be a locally irreducible decomposition of $\{z|\nu(u,z)\ge1\}$ at $o$ and all $A_k$ are irreducible in some neighborhood $U_0$ of $o$.
Note that $1<\nu(u,o)\le2$, by Theorem \ref{Siu_decomposition}, we obtain that $(a)$ $m=1,\ (A_1,o)$ is regular or singular with multiplicities 2; or $(b)$ $m=2,\ (A_k,o)\ (k=1,2)$ are all regular.

Case (i). $(\{z|\nu(u,z)\ge1\},o)$ is regular. Thus, we have $$(N(\mathscr{I}(u)),o)=(\{z|\nu(u,z)\ge1\},o).$$ Hence, in some local coordinates, it follows that $\mathscr{I}(u)_{o}{=}(z_1)\cdot\mathcal{O}_2$ by Lemma \ref{Guan}.

Case (ii). $(\{z|\nu(u,z)\ge1\},o)$ is singular with multiplicities 2. Let $f(z_1,z_2)\in\mathcal{O}_2$ be the minimal defining function of the germ of analytic set $(\{z|\nu(u,z)\ge1\},o)$. As $(\{z|\nu(u,z)\ge1\},o)$ is singular at $o$ with multiplicities 2, by Weierstrass Preparation Theorem, in some local coordinates we can assume $$f(z_1,z_2)=z_1^2+a_1(z_2)z_1+a_2(z_2),$$
where $a_1(0)=a_2(0)=0$.

Note that $$f(z_1,z_2)=\left(z_1+\frac{a_1(z_2)}{2}\right)^2+a_2(z_2)-\left(\frac{a_1(z_2)}{2}\right)^2.$$
Let $a(z_2):=a_2(z_2)-(\frac{a_1(z_2)}{2})^2$. We can compute the log canonical threshold $c_{o}(f)$ as follows:

For ord$_oa(z_2)=\infty$, i.e., $a(z_2)\equiv0$, we have $c_{o}(f)=\frac{1}{2}$.

For ord$_oa(z_2)=2k,\ k\ge1$, we have $c_{o}(f)=\frac{1}{2}+\frac{1}{2k}$.

For ord$_oa(z_2)=2k-1,\ k\ge1$, we have $c_{o}(f)=\frac{1}{2}+\frac{1}{2k-1}$.

It is easy to see that $c_{o}(f)=1$ if and only if $f(z_1,z_2)=z_1z_2$ in some local coordinates, i.e., $o$ is a simple normal crossing singularity of $f(z_1,z_2)$. As $c_{o}(\mathscr{I}(u))\ge c_{o}(u)=1$, we have $(N(\mathscr{I}(u)),o)=(N(z_1z_2),o)$ or $(N(z_1),o)$. By Lemma \ref{Guan}, it follows that $\mathscr{I}(u)_{o}$ is $(z_1z_2)\cdot\mathcal{O}_{2}$ or $(z_1)\cdot\mathcal{O}_{2}$ (In fact, here only the normal crossing case occurs by Theorem \ref{G-Z}). \hfill $\Box$

\begin{rem}
Note that $1\le\nu(u,o)\le2$. It is easy to see that the multiplier ideal sheaves cannot be uniquely determined by Lelong number and log canonical threshold.

For $\nu(u,o)=2$, $\mathscr{I}(u)_{o}$ can be any one of the three case, e.g. we can take $u=2\log(|z_1|+|z_2|)$ or $\log|z_1|+\log(|z_1|+|z_2|)$ or $\log|z_1z_2|$, respectively.

For $1<\nu(u,o)<2$, $\mathscr{I}(u)_{o}$ can be any one of the first two cases, e.g., we can take $u=\log(|z_1|^\alpha+|z_2|^\beta)$ with $\alpha>0,\beta>2,\alpha^{-1}+\beta^{-1}=1$ or $\log|z_1|+\alpha\log(|z_1|+|z_2|)\ (0<\alpha<1)$, respectively.

For $\nu(u,o)=1$, $\mathscr{I}(u)_{o}=(z_1)\cdot\mathcal{O}_{2}$ by Theorem \ref{G-Z}.
\end{rem}

\begin{rem}
For the two dimensional case, we can say nothing about the multiplier ideal sheaves with weights of log canonical threshold less than one.

As all integrally closed ideals are multiplier ideal sheaves in two dimensional regular local rings \cite{F-M}, then by Zariski's decomposition theorem every multiplier ideal $\mathcal{I}$ in $\mathcal{O}_2$ can be written as
$$\mathcal{I}=f_1^{l_1}{\cdots}f_m^{l_m}\mathcal{I}_1^{k_1}{\cdots}\mathcal{I}_n^{k_n},$$
where $\mathcal{I}_1, ..., \mathcal{I}_n$ are simple $\mathfrak{m}_2$-primary integrally closed ideals, $f_1, ..., f_m$ are pairwise relatively prime irreducible elements of $\mathcal{O}_{2}$, and $l_1, ..., l_m, k_1, ..., k_n$ are positive integers. However, it is still difficult to characterize the simple $\mathfrak{m}_2$-primary integrally closed ideals.
\end{rem}

In addition, we have the following result for the jumping number one case in arbitrary dimension.

\begin{prop} \label{jumber}
For any non-trivial multiplier ideal $\mathscr{I}_o\subset\mathcal{O}_n$, there exists a plurisubhamonic function $u$ near $o$ and an ideal $I\subset\mathcal{O}_n$ such that the jumping number $c_o^I(u)=1$ and $\mathscr{I}(u)_o=\mathscr{I}_o$.
\end{prop}

\begin{proof}
Let $\mathscr{I}_o=\mathscr{I}(\varphi)_o$ for some plurisubharmonic function $\varphi$ near $o$. It follows from Guan-Zhou's solution to Demailly's strong openness conjecture, there is a plurisubharmonic function $\varphi_A$ with analytic singularities near $o$ such that $\mathscr{I}(\varphi_A)_o=\mathscr{I}(\varphi)_o$.

In addition, by a discussion on Hironaka's log resolution \cite{De15}, there is an increasing discrete sequence $(\xi_k)_{k\in\mathbb{N}}$ which goes to infinity with $\xi_0=0$ such that
$$\mathscr{I}(c{\cdot}\varphi_A)_o=\mathscr{I}(\xi_k{\cdot}\varphi_A)_o\supsetneqq\mathscr{I} (\xi_{k+1}{\cdot}\varphi_A)_o,$$
for every $c\in[\xi_k,\xi_{k+1})$ and every $k$.

As the increasing sequence $(\xi_k)$ goes to infinity, there exists $k_0\geq1$ such that $1\in[\xi_{k_0},\xi_{k_0+1})$. Take $u=\xi_{k_0}\varphi_A$ and $I=\mathscr{I}(\xi_{k_0-1}\varphi_A)_o$. Then, we have\\

\hspace*{\fill}
$\mathscr{I}(u)_o=\mathscr{I}(\varphi_A)=\mathscr{I}_o \mbox{ and }c_o^I(u)=c_o^I(\xi_{k_0}\varphi_A)=1.$
\hspace*{\fill}
\end{proof}

\subsection{Proof of Theorem \ref{main2}}
To prove Theorem \ref{main2}, the following Ohsawa-Takegoshi $L^2$ extension theorem is necessary.

\begin{thm} \emph{(\cite{O-T}).} \label{O-T}
Let $D$ be a bounded pseudoconvex domain in $\mathbb{C}^n$. Let $u$ be a plurisubharmonic function on $D$. Let $H$ be an $m$-dimensional complex plane in $\mathbb{C}^n$. Then for any holomorphic function on $H{\cap}D$ satisfying
$$\int_{H{\cap}D}|f|^2e^{-2u}d\lambda_H<+\infty,$$ there exists a holomorphic function $F$ on $D$ such that $F|_{H{\cap}D}=f$ and $$\int_D|F|^2e^{-2u}d\lambda_n\leq C_D\cdot\int_{H{\cap}D}|f|^2e^{-2u}d\lambda_H,$$
where $d\lambda_H$ is the Lebesgue measure, and $C_D$ is a constant which only depends on the diameter of $D$ and $m$.
\end{thm}

\noindent{\textbf{\emph{Proof of Theorem} \ref{main2}.}
By the restriction formula about log canonical threshold, we have $$c_{o}(u){\ge}c_{o'}(u|_H)=1.$$

Case (i). $c_{o}(u)>1$. By the definition of log canonical threshold, it follows that $\mathscr{I}(u)_{o}=\mathcal{O}_{n}$.

Case (ii). $c_{o}(u)=c_{o'}(u|_H)=1$. Let $A:=N(\mathscr{I}(u))$. If $(A{\cap}H,o)$ is regular, by Theorem 2.3 in \cite{G-Z_restriction}, we have $\mathscr{I}(u)_{o}=(z_1,z_2)\cdot\mathcal{O}_{n}$ or $(z_1)\cdot\mathcal{O}_{n}$, up to a change of variables at $o$.

Note that $(A{\cap}H,o)\subset(N(\mathscr{I}(u|_H)),o)$ by the restriction formula about multiplier ideal sheaves. If $(A{\cap}H,o)$ is singular, then $(A{\cap}H,o)=(N(z_1z_2|_H),o)$ by Theorem \ref{main}. Following from Theorem \ref{O-T}, we have a holomorphic function $F(z)$ on a Stein neighborhood $U$ of $o$ such that $$F(z)|_{U{\cap}H}=z_1z_2|_{U{\cap}H}\ \mbox{and}\ F(z)_{o}\in\mathscr{I}(u)_{o}.$$
In particular, $(A,o)\subset(N(F(z)),o)$.
\\

\textbf{Claim.} $(A,o)$ is of pure dimension $n-1$ and $\mathscr{I}(u)_{o}=F(z)\cdot\mathcal{O}_{n}$.

Suppose that $(A,o)=(A_{n-1},o){\cup}(A_{\le n-2},o)$ is a local decomposition such that $(A_{n-1},o)$ is of pure dimension $n-1$ and $\dim(A_{\le n-2},o){\le}n-2$. 

By Corollary 3.5 in \cite{G-Z_restriction}, for almost every point $x_0{\in}A{\cap}H$ near $o$, $c_{x_0}(u|_H)=c_{x_0}(u)=1$. By Theorem 2.2, \cite{G-Z_restriction}, we have $\dim_{x_0}A=n-2+\dim_{x_0}A{\cap}H=n-1$ for any point $x_0$ near $o$ with $c_{x_0}(u|_H)=c_{x_0}(u)=1$. Then, it follows that for almost every point $x_0{\in}A{\cap}H$ near $o$, $x_o\not\in A_{\le n-2}$, which implies $A_{\le n-2}\cap H=o$ and $$(A_{n-1}{\cap}H,o)=(N(z_1z_2|_H,o))=(N(F(z)){\cap}H,o).$$

As $F(z)|_{U{\cap}H}=z_1z_2|_{U{\cap}H}$, $F(z)_{o}$ is the minimal defining function of $(N(F(z)),o)$. Otherwise, there exists holomorphic function $g$ near $o$ such that $g^2|F$, which contradicts to $F(z)|_{U{\cap}H}=z_1z_2|_{U{\cap}H}$. Let $(F_{n-1}(z),o)$ be the minimal defining function of $(A_{n-1},o)$, which is a factor of $F(z)_{o}$. Suppose that $F(z)$ has another factor $h(z)$ near $o$ such that $(N(h(z)),o)\not\subset(A_{n-1},o)$, then $F_{n-1}(z)h(z)|F(z)$, which implies that $F(z)|_{U{\cap}H}$ has at least order 3 at $o$. It contradicts to $F(z)|_{U{\cap}H}=z_1z_2|_{U{\cap}H}$. Hence, $F(z)_{o}$ is the minimal defining function of $(A_{n-1},o)$. Thus, it follows from $(A,o)\subset(N(F(z)),o)=(A_{n-1},o)$ that
$$(A_{n-1},o)=(N(F(z)),o)=(A,o)\mbox{ and } \mathscr{I}(u)_{o}=F(z)\cdot\mathcal{O}_{n}.$$

By Weierstrass Preparation Theorem, we can assume that
$$F(w_1,w_2,z_3,...,z_n)=w_1^2+a_1(w_2,z_3,...,z_n)w_1+a_2(w_2,z_3,...z_n)$$
and $F(w_1,w_2,z_3,...,z_n)|_H=w_1^2+w_2^2$.

Note that $$F(w_1,w_2,z_3,...,z_n)=\left(w_1+\frac{a_1}{2}\right)^2+a_2-\left(\frac{a_1}{2}\right)^2.$$
Let $a_3(w_2,z_3,...,z_n):=a_2-\left(\frac{a_1}{2}\right)^2$. As $F(w_1,w_2,z_3,...,z_n)|_H=\omega_1^2+\omega_2^2$, we have $a_3|_H=w_2^2$. Apply Weierstrass Preparation Theorem to $a_3$, we assume $$a_3(w_2,z_3,...,z_n)=e(w_2,z_3,...,z_n)\left(w_2^2+b_1(z_3,...,z_n)w_2+b_2(z_3,...,z_n)\right),$$ where $e(w_2,z_3,...,z_n)\in\mathcal{O}_{n}$ is a unit and $b_1(0)=b_2(0)=0$.

Let $b_3:=b_2-\left(\frac{b_1}{2}\right)^2$. Then we obtain that
$$F(w_1,w_2,z_3,...,z_n)=e(w_2,z_3,...,z_n)\left(e(w_2,z_3,...,z_n)^{-1}w_1^2+(w_2-\frac{b_1}{2})^2+b_3\right).$$
Change the variables at $o$ again, it follows that
$$F(w_1,w_2,z_3,...,z_n)=e(w_2,...,z_n)\left(w_1^2+w_2^2+f(z_3,...,z_n)\right),$$
where $e(0)\neq0,\ f(z_3,...,z_n)\in\mathfrak{m}_{n}$. That is to say\\

\hspace*{\fill}$\mathscr{I}(u)_{o}=\left(w_1^2+w_2^2+f(z_3,...,z_n)\right)\cdot\mathcal{O}_{n}.$\hspace*{\fill} $\Box$

\vspace{.1in} {\em Acknowledgements}. Both authors would like to sincerely thank our supervisor, Professor Xiangyu Zhou for valuable discussions and useful comments.

The second author is also grateful to Beijing International Center for Mathematical Research for its good working conditions.

\end{document}